\documentclass[12pt, reqno]{article}

\usepackage{fullpage}
\usepackage{enumerate}
\usepackage{setspace}
\usepackage{graphicx}
\usepackage{subcaption}
\usepackage{bbm}

\usepackage{upgreek, mathtools,bm}

\usepackage{amsmath,amsfonts,amssymb,graphics,amsthm, comment}
\usepackage{hyperref}
\hypersetup{
    colorlinks=true,
    linkcolor=blue,
    citecolor=red,
    urlcolor=blue,
    pdfborder={0 0 0}
}
\usepackage{yfonts}
\usepackage[normalem]{ulem}
\usepackage{graphicx}
\usepackage{xcolor}
\usepackage{bbm}
\usepackage{wrapfig} 

\usepackage[font=sf, labelfont={sf,bf}, margin=1cm]{caption}

\usepackage{cleveref}
  \crefname{theorem}{Theorem}{Theorems}
  \crefname{thm}{Theorem}{Theorems}
  \crefname{thm*}{Theorem*}{Theorems}
  \crefname{lemma}{Lemma}{Lemmas}
  \crefname{lem}{Lemma}{Lemmas}
  \crefname{remark}{Remark}{Remarks}
  \crefname{prop}{Proposition}{Propositions}
\crefname{notation}{Notation}{Notations}
\crefname{claim}{Claim}{Claims}
  \crefname{defn}{Definition}{Definitions}
  \crefname{corollary}{Corollary}{Corollaries}
  \crefname{section}{Section}{Sections}
  \crefname{figure}{Figure}{Figures}
    \crefname{assumption}{Assumption}{Assumptions}

\newtheorem{thm}{Theorem}[section]
\newtheorem{thm*}{Theorem*}[section]

\newtheorem{lemma}[thm]{Lemma}
\newtheorem{corollary}[thm]{Corollary}
\newtheorem{prop}[thm]{Proposition}
\newtheorem{defn}[thm]{Definition}
\newtheorem{conj}[thm]{Conjecture}

\numberwithin{equation}{section}

\theoremstyle{definition}
\newtheorem{remark}[thm]{Remark}

\def\cN{\mathcal{N}}
\def\cM{\mathcal{M}}

\def\cH{\mathcal{H}}

\def\cE{\mathcal{E}}

\def \ve {\varepsilon}

\def\P{\mathbb{P}}

\def\E{\mathbb{E}}

\def\R{\mathbb{R}}
\def\Z{\mathbb{Z}}
\def\S{\mathbb{S}}

\def  \p- {p\textunderscore}

\def\ep{\varepsilon}

\newcommand{\abs}[1]{ \left \lvert #1 \right \rvert}

\newcommand{\de}{\delta}

\newcommand{\lr}[1]{\left( {#1}\right)}

\newcommand{\la}[1]{ \left\lvert {#1}\right\rvert}
\newcommand{\set}[1]{ \left\{ #1\right\}}

\newcommand{\inn}[1]{ \left\langle #1\right\rangle}

\newcommand{\1}{\mathbf{1}}

\newcommand{\Ls}{\mathcal{L}}

\newcommand{\tnrm}[1]{{\left\vert\kern-0.25ex\left\vert\kern-0.25ex\left\vert #1 
    \right\vert\kern-0.25ex\right\vert\kern-0.25ex\right\vert}}

\begin{document}
 \title{An upper bound on the smallest singular value of dense random combinatorial matrices}

\author{\textsc{Dongbin Li, Alexander E. Litvak,  and Tingzhou Yu}}
\date{}

\maketitle
\abstract{
Let $M$ be an $n\times n$ random matrix with entries in $\{0, 1\}$, where each row is independently and uniformly sampled from the set of all vectors in $\{0, 1\}^n$ containing exactly $d$ ones, with $d=pn$ for some fixed constant $p\in (0,1/2]$. A recent result of Tran states that the smallest singular value $s_n(M)$ is bounded below by $c_p n^{-1/2}$ with high probability. In this note, we establish a complementary upper bound for $s_n(M)$, proving that \[
  \forall \varepsilon >0 \qquad \P\lr{s_n(M)\le \frac{\sqrt{d}}{\varepsilon^2 n}}\ge 1-C_p\lr{\varepsilon+\frac{1}{\sqrt{d}}}, 
\]where $C_p$ is a positive constant depending only on $p$. This result confirms that the least singular value $s_n(M)$ of dense random combinatorial matrices is typically of the order $n^{-1/2}$.}

\bigskip

{\small
\noindent{\bf AMS 2010 Classification:}
primary: 60B20, 15B52, 46B06; 
secondary: 60C05, 05C80, 46B09.\\
\noindent
{\bf Keywords:} condition number, invertibility of random matrices,
random graphs, random matrices, regular graphs, singular probability, 
smallest singular value.}


\section{Introduction}\label{sec:intro}

Let $M$ be an $n\times n$ matrix with real entries,  recall that the smallest singular value of $M$, denoted by $s_n(M)$, can be defined as
\begin{equation*}
    s_n(M):=\inf_{x\in \mathbb{S}^{n-1}}\|Mx\|_2,
\end{equation*}
where $\mathbb{S}^{n-1}:=\{x\in \R^n: \|x\|_2=1\}$ is the unit sphere in $\R^n$. Note that $s_n(M)>0$ if and only if $M$ is invertible. When $M$ is invertible, $s_n(M)=1/\|M^{-1}\|$, where $\|M\|=\sup_{\|x\|_2=1}\|Mx\|_2$ is the operator norm of  $M : \ell_2^n \to \ell_2^n$.
The study of the smallest singular value is of fundamental importance due to its crucial role in theoretical questions 
and practical applications. In numerical analysis, the smallest singular value is closely related to the condition number, 
which measures the worst-case precision loss in computational problems \cite{spielman2004smoothed,tikhomirov2022quantitative}. 
Beyond numerical stability, the smallest singular value is essential in a variety of contexts, including the analysis of matrix 
singularity probabilities (for 0/1 matrices with independent entries see, for example, \cite{basak2021sharp,jain2020sharp,jain2020sharpsing,
litvak2022singularity,tikhomirov2020singularity}), and the study of the empirical spectral distribution of random matrices (see \cite{TVuniversality, bordenave2012around, RTcirclawmin} and references therein for the historical account of the problem).

 
The quantitative study of the smallest singular value of random matrices dates back to von Neumann and his collaborators, who 
conjectured that the smallest singular value is of order $n^{-1/2}$ with high probability (see \cite{MR157875, bams/1183511222}).
This estimate was later confirmed by Edelman  \cite{edelman1988eigenvalues} and Szarek  \cite{szarek1991condition}
  for random matrices with independent identically distributed (i.i.d.) standard Gaussian entries.
In  \cite{rudelson2008least, rudelson2008littlewood}, Rudelson and Vershynin established the sharp bound $\Theta(n^{-1/2})$ 
for square matrices with i.i.d. subgaussian entries, that is, $c n^{-1/2}\leq s_n(M)\leq C n^{-1/2}$. 
A similar behaviour for square matrices with i.i.d. entries possessing a finite second moment is obtained by combining the results by Rebrova and Tikhomirov \cite{rebrova2018coverings} and Tatarko \cite{tatarko2018upper}. For other results and bounds on the smallest singular value of square random matrices with independent entries, we refer the reader to 
\cite{kahn1995probability,   livshyts2021smallestTK,  tao2009inverse, tao2010random}. 
The study of the corresponding bounds for the rectangular matrices with i.i.d. subgaussian 
variables was started in \cite{litvak2005smallest} and further developed in 
\cite{rudelson2009smallest, livshyts2021smallest,  mendelson2014singular}.

In contrast, estimating the smallest singular value becomes considerably more challenging when dependencies exist between 
the matrix entries, and progress has been comparatively much slower. For example, 
only very recently  a sharp lower tail bound for the least singular value of $n \times n$ symmetric matrices with i.i.d. 
subgaussian entries has been obtained by Campos, Jenssen, Michelen and Sahasrabudhe \cite{campos2024least}, despite the analogous 
result for non-symmetric matrices being established nearly 17 years earlier by Rudelson and Vershynin \cite{rudelson2008littlewood}.
In recent years, considerable attention has also been given to models derived from combinatorics and graph theory, beyond symmetric 
random matrices. In particular, estimates of the least singular value of such models have been studied extensively  (see, for example, \cite{Cooksingularitydreg,jain2021approximate,
jain2022smallest,
litvak2017adjacency, litvak2019smallest, tran2020smallest} for some recent developments in this area). 
Our work focuses on one of such models. 

Let $d\leq  n$ be positive integers. We consider the set $\cM_{n,d}$  of all $n\times n$ matrices with entries in $\{0,1\}$, 
where each row is chosen independently and uniformly from the set of $\{0, 1\}^n$ vectors with exactly $d$ ones. An element 
$M\in \cM_{n,d}$ can be interpreted as the adjacency matrix of a directed graph (digraph) on $n$ labeled vertices, where each 
vertex has exactly $d$ outgoing edges. Alternatively, $M$ can be viewed as the adjacency matrix of a bipartite graph between two
disjoint sets of $n$ vertices each, with each vertex on the left having degree $d$. Note that while the out-degree (row sum) of 
each vertex is fixed at $d$, the in-degree (column sums) are random with mean $d$. This suggests that theoretically our model can 
produce matrices with zero columns, especially when $d$ is small. 
In fact, it is not difficult to verify  that for any (fixed) $\ep >0$ and 
$d \leq (1-\ep) \log n$, the random matrix $M$ asymptotically almost surely contains a zero column \cite{APsparseRCM, LLY-circ}. 
Moreover, the threshold $d=\log n$ 
 is sharp in the sense that for any given $\ep >0$, if $\min{(d,n-d)} \geq (1+\ep) \log n$, then a matrix $M$ chosen 
 uniformly at random from $\cM_{n,d}$ is asymptotically almost surely invertible, a result established by Ferber, Kwan, 
 and Sauermann \cite[Theorem 1.2]{ferber2022singularity}. 
It is noteworthy that this sharp threshold for singularity also arises in the context of random Bernoulli matrices (see, e.g., \cite{addario2014hitting, basak2021sharp, litvak2022singularity}).

The behavior of the least singular value of this model in the dense regime, specifically for degree $d=n/2$ was first studied 
in \cite{nguyen2013singularity,NVcircgivensum}. Nguyen and Vu showed that for any $C>0$, there exists $D>0$ such that 
\[
\mathbb{P}(s_n(M) < n^{-D}) \leq n^{-C}.
\]
Building on the earlier work of Ferber, Jain, Luh, and Samotij \cite{ferber2021counting}, Jain \cite{jain2021approximate} improved this bound, showing that 
\[
\mathbb{P}(s_n(M) < \varepsilon n^{-2}) \leq C \varepsilon + 2 e^{-n^{0.0001}}. 
\]
More recently, Tran \cite{tran2020smallest} showed that there exist constants $C, c>0$ such that for all $\varepsilon\ge 0$,
\[
\P(s_n(M)\le \varepsilon n^{-1/2})\le C\varepsilon+2e^{-cn}.
\]
Further refinements to the probability bounds for $s_n(M)$ when $d=n/2$ were obtained in \cite{jain2020sharp}. It was shown that for every $\varepsilon>0$, there exists $C=C(\varepsilon)$  such that for all $t\ge 0$,
\[
\P(s_n(M)\le t n^{-1/2})\le Ct+(1/2+\varepsilon)^n. 
\]
The result also confirms that the singularity probability of $M$ is $(1/2+o(1))^n$, as conjectured by Nguyen in \cite{nguyen2013singularity}.

Although the existing lower bounds provide insight into the anticipated scale of $s_n(M)$,
a matching probabilistic upper bound has not been obtained for the $\cM_{n, d}$ model in the dense regime. The main contribution of this paper is to establish this missing upper bound, thereby determining the sharp asymptotic order of $s_n(M)$ for this class of random matrices. Specifically, we prove the following main result in this note.

\begin{thm}\label{thm:upper_bdd_least}
 Let $p \in (0, 1/2]$ be a fixed constant and let $d, n$ be integers such that $d = pn$. Let $M$ be a random $n\times n$ matrix 
 uniformly drawn from      $\cM_{n,d}$.
Then there exists a constant $C_p>0$ depending only on $p$ such that for every $\varepsilon>0$,
    \[
    \P\lr{s_n(M)\le \frac{\sqrt{d}}{\varepsilon^2 n}}\ge 1-C_p\lr{\varepsilon+\frac{1}{\sqrt{d}}}.
    \]
\end{thm}

\begin{remark}
Theorem~\ref{thm:upper_bdd_least} shows that in the dense regime ($d=pn$, $p$ is fixed) the upper bound  $s_n(M)=O(\sqrt{d}/n)$ holds with high probability, in particular, it implies that $s_n(M)=O_p(n^{-1/2})$.
The matching lower bound on $s_n(M)$ was proved in \cite{tran2020smallest} (see the third part of the remark following Theorem~1.2 
there, see also \cite{jain2020sharp} for $d=n/2$). Thus, our upper bound confirms that in the dense regime 
the typical order of the least singular value $s_n(M)$ is of order $n^{-1/2}$. 
\end{remark}

\begin{remark}
By Theorem~4.1 in \cite{LLY-circ}, for some absolute constant $C>0$,   $\|M-\E M\|\leq C\sqrt{d}$ with probability 
at least $1-2/n$, whenever $d\geq \log n$ (for $d=n/2$ this bound with better probability was obtained in 
\cite[Proposition 2.8]{tran2020smallest} (see also \cite{jain2020sharp}) and  note that the proof in 
\cite{tran2020smallest} can be extended to $d=pn$). Since all entries of $\E M$ are $d/n$, its norm is $d$. Therefore, 
by the triangle inequality,  $\|M\|\leq 2d $ with probability at least $1-2/n$ for large enough $d$.  Note that 
 $M \1_n=d\1_n$, where $\1_n=(1,1,\dots,1)$, hence $\|M\| \geq d$. Thus, the largest singular value of $M$, $s_1(M)=\|M\|$ is of order $d$, therefore, 
 in the dense regime $d=pn$,  the condition number 
 $$\kappa(M):=s_{1}(M)/s_{n}(M) \approx n^{3/2}$$ 
 (equivalence is up to constants depending only on $p$  and with probability at least $1-C_p'/\sqrt{n}$). 
 This has important implications for numerical computation. Indeed, it is well known that condition number is closely linked to computational complexity, particularly in the context of numerical linear algebra and iterative methods. In general, a large condition number typically leads to slower convergence (see, e.g., \cite[Chapters~2 and 11]{GVanMatrix}).
 We would also like to mention that in general the operator norm of $M$ could be as large as its Hilbert-Schmidt norm, 
 $\sqrt{dn}$, as the example of matrix having first  $d$ columns equal to $\1_n$ and $n-d$ zero columns  shows 
(in this case the norm is attained on the vector $e_{1}+e_{2}+...+e_{d}$).  
\end{remark}

\begin{remark}
It is also natural to consider the behavior of $s_n(M)$ in the \textit{sparse} regime, where $C_0\log n \leq d=o(n)$ for an absolute constant $C_0>1$. 
For such sparse random matrices, it is speculated that the typical order of $s_n(M)$ is $\sqrt{d}/n$ (see also a similar conjecture in \cite{litvak2019smallest} 
 for $d$-regular matrices).  Notably, this conjectured order $\sqrt{d}/n$ appears in the upper bound stated in \cref{thm:upper_bdd_least}. However, despite the  well-understood by now behavior of the least singular value in the dense case, its sparse counterpart has remained far less accessible. The only known  quantitative lower bound for $s_n(M)$ has been recently established in \cite{LLY-circ}, where the authors proved that for large enough $n$, some absolute positive  constants $C$ and $c$, and for   $d\geq \log^{2} n$, 
\[
   \P(s_n(M)\le  n^{-c})\le C/\sqrt{d} .
\] 
 While the expected order $\sqrt{d}/n$ remains out of reach in the sparse regime, our numerical simulations offer supporting evidence for the conjecture. Figures \ref{fig:sim_power} and \ref{fig:sim_log} present log-log plots of the mean $s_n(M)$ versus $n$ for  the intermediate regimes $d=n^{1/3}$ and $d=5\log n$, respectively. In both cases, the observed data aligns well with a reference line proportional to $\sqrt{d}/n$, indicating that this order likely holds in the sparse regime as well.
\end{remark}

\begin{figure}[htbp]
    \centering 
    \begin{subfigure}[b]{0.48\textwidth} 
        \includegraphics[width=\linewidth]{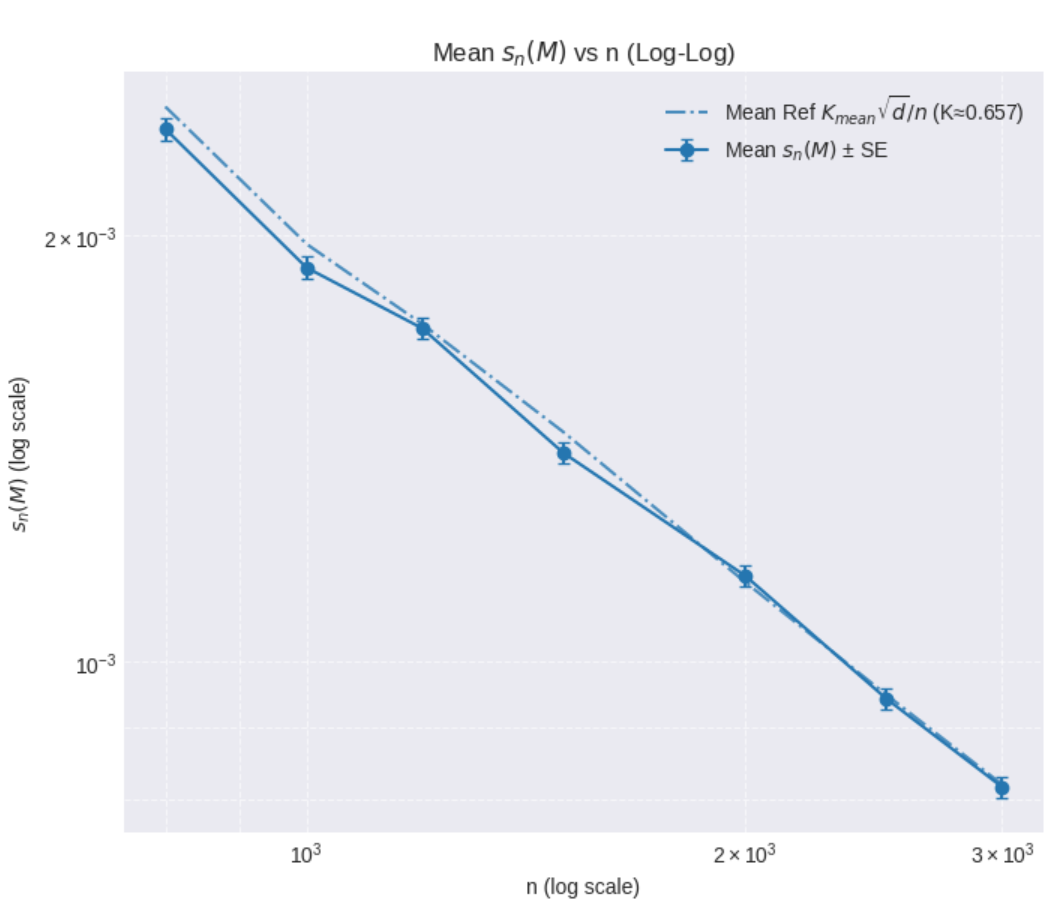} 
        \caption{Simulation for $d=\lfloor n^{1/3} \rfloor$} 
        \label{fig:sim_power} 
    \end{subfigure}
    \hfill 
    \begin{subfigure}[b]{0.48\textwidth} 
        \includegraphics[width=\linewidth]{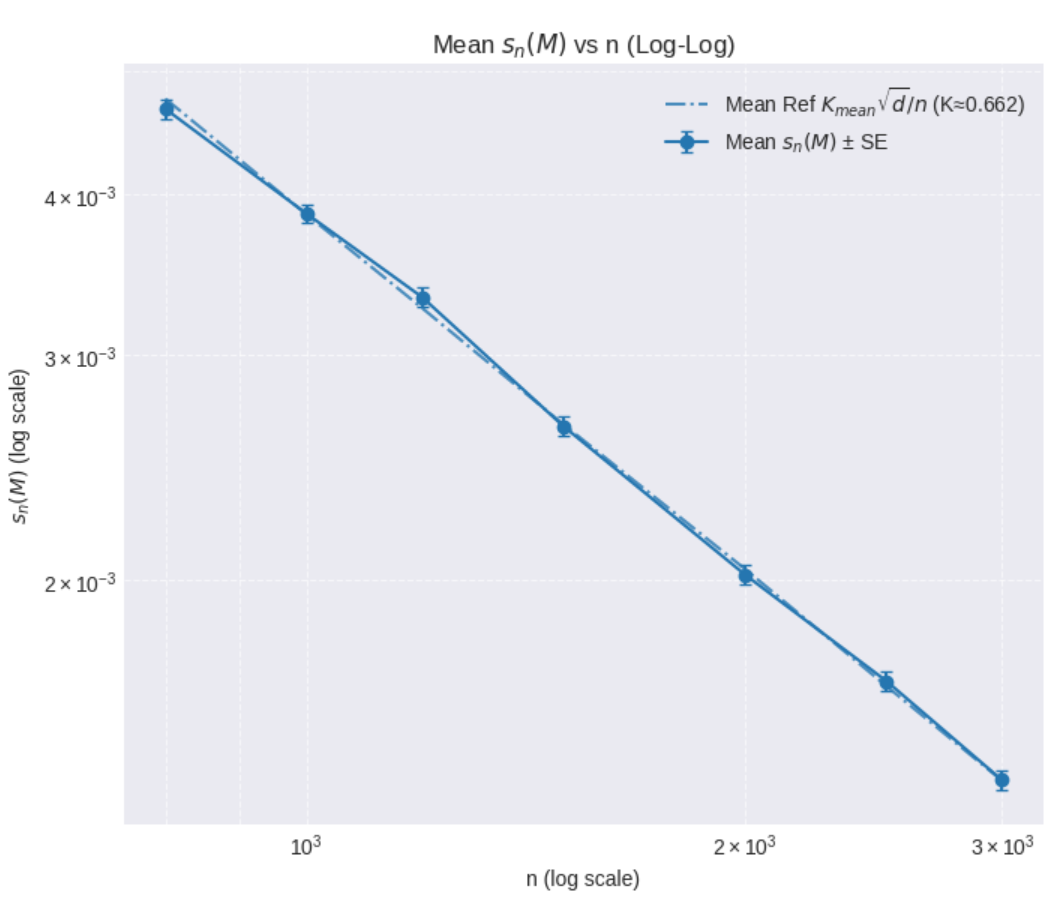} 
        \caption{Simulation for $d=5 \lfloor \log n \rfloor$} 
        \label{fig:sim_log} 
    \end{subfigure}

    \caption{Log-log plots of the mean smallest singular value $\mathbb{E}[s_n(M)]$ versus $n$. Points show simulation results (mean $\pm$ standard error). The dashed lines represent reference slopes proportional to $\sqrt{d}/n$. (a) Case $d=\lfloor n^{1/3} \rfloor$. (b) Case $d=5 \lfloor \log n \rfloor$.} 
    \label{fig:simulations} 
\end{figure}

\begin{conj} \label{conjupper}
    For every $\ve, \delta >0$ there is a constant $C=C(\ve, \de)$ 
    depending only on $\ve$, $\de$ such that the following holds. 
     Let $d\le n$ be integers that satisfy $\min(d,n-d) \geq (1+\ve) \log n$. Let $M$ be a random $n\times n$  matrix uniformly drawn from  $\cM_{n,d}$.
  Then   
     \[\P\lr{s_n(M) > C\, \frac{\sqrt{d}}{ n}} \leq \delta.
    \]
\end{conj}
\begin{remark}
    For fixed $\ve>0$, and $d \leq (1-\ve) \log n$, it is not difficult to check that with high probability, $M$ contains a zero column. Therefore, $s_n(M)=0$ with high probability whenever $\min(d,n-d) \leq (1-\ve) \log n$ and Conjecture \ref{conjupper} holds in this regime.
\end{remark}

\paragraph{Organization of the paper.}  This paper is structured as follows. In \cref{sec: preli}, we introduce essential preliminaries, including key concepts such as biorthogonal systems, almost constant vectors, and concentration inequalities. In \cref{sec: inv_almost}, we establish the invertibility of the matrix $M$ on almost constant vectors. The proof of \cref{thm:upper_bdd_least} is presented in \cref{sec: proof_main}.


\section{Preliminaries}\label{sec: preli}

\subsection{Notation.} We use the following standard notations.
Let $\{e_1,e_2,\dots, e_n\}$ be the canonical basis of $\R^n$ 
equipped with the canonical inner product $\langle\cdot,\cdot \rangle$ 
and the canonical Euclidean norm  $\|\cdot \|_2$.  
We use $[n] := \{1, 2, \dots, n\}$ to represent the set of the first 
$n$ positive integers.  
The vector $\mathbf{1}_n \in \mathbb{R}^n$ denotes the vector with each 
component equals $1$. 
The set $$\mathbb{S}^{n-1} := \{x \in \mathbb{R}^n : \|x\|_2 = 1\}$$ represents the  Euclidean unit sphere in $\mathbb{R}^n$. 
Matrices are denoted by uppercase letters, e.g., $M$, with $M^T$ indicating 
the transpose of $M$. For a matrix $M$, $R_i=R_i(M)$ denotes its $i$-th row, 
and $X_i := R_i^T$ is the corresponding column vector. 
The distance from a vector $v$ to a subspace $H$ is denoted by 
$$\operatorname{dist}(v, H) := \inf_{w \in H} \|v - w\|_2.$$ 
The operator norm of an $m\times k$  matrix $A$, considered as an operator acting between 
corresponding Euclidean spaces (that is, $A: \ell_2^k\to \ell_2^m$) is denoted by $\|A\|$. Respectively, 
given a subspace $V\subset \R^k$, $\|A_{|V}\|$ denotes the operator norm of $A$ restricted to $V$. 
By $C, c, C_0, C_1, \cdots$ we denote either positive absolute constants independent of all parameters or positive constants that may depend on the parameter $p=d/n$. They  may vary from line to line.

\subsection{Biorthogonal systems}
Biorthogonal systems play an important role in the analysis of 
 the invertibility of matrices and the structure of their inverses (see, e.g., \cite{rudelson2008least, tatarko2018upper}).
 Let $\{E_k\}_{k=1}^m$ and $\{F_k\}_{k=1}^m$ be two sets of vectors in a Hilbert space $\cH$ with inner product 
 $\inn{\cdot, \cdot}$.  We say that the system $\{E_k, F_k\}_{k=1}^m$  is a \textit{biorthogonal system} in $\cH$ 
 if it satisfies $\langle E_i, F_j\rangle=\delta_{ij}$ for all $i, j\in [m]$, where $\delta_{ij}$ is the Kronecker delta. 
 Note that if $\{E_k, F_k\}_{k=1}^m$  is a biorthogonal system, then 
 both $\{E_k\}_{k=1}^m$ and $\{F_k\}_{k=1}^m$ are linearly independent. Thus, if the dimension of $\cH$ is $m$ then 
 $\{E_k\}_{k=1}^m$ spans $\cH$ as well as $\{F_k\}_{k=1}^m$.   We summarize key properties of biorthogonal systems (see, e.g., \cite[Proposition 2.1]{rudelson2008least}) relevant to the proof of our main result in the following proposition. 

\begin{prop}\label{prop_biorth} Let  $\cH$ be an         $n$-dimensional Hilbert space. 
    \begin{enumerate}
        \item For any invertible $n\times n$ matrix $A$, the system $\{Ae_k, (A^{-1})^Te_k\}_{k=1}^n$ forms a  biorthogonal system in $\R^n$. 
        \item Given a  linearly independent set of vectors $\{E_k\}_{k=1}^n$ in an 
        $n$-dimensional Hilbert space $\cH$, there exist unique vectors $\{F_k\}_{k=1}^n$ such that $\{E_k, F_k\}_{k=1}^n$ is a  biorthogonal system in $\cH$. 
        \item For a  biorthogonal system $\{E_k, F_k\}_{k=1}^n$ in an 
        $n$-dimensional Hilbert space $\cH$, the Euclidean norm of each $F_k$ satisfies
        \begin{equation}\label{eq:bioth}
            \|F_k\|_2=\frac{1}{\mbox{dist}(E_k, H_k)}, \qquad k=1, \dots, n,
        \end{equation}
        where $H_k=\mbox{span}\{E_i\}_{i\neq k}$ denotes the subspace spanned by all $E_i$ except $E_k$, and $\operatorname{dist}(E_k, H_k)$ is the Euclidean distance from $E_k$ to $H_k$.
    \end{enumerate}
\end{prop}

\subsection{Almost constant vectors}

To obtain an upper bound on the smallest singular value $s_n(M)$, we consider a decomposition of the unit sphere $\S^{n-1}$ 
into \textit{almost constant} and \textit{non-almost constant vectors}. The notion of almost constant vectors was introduced 
in a similar context in earlier works  \cite{Cooksingularitydreg, litvak2017adjacency} and then developed in \cite{litvak2019smallest, litvak2022singularity}. 

\begin{defn}\label{def:almost_constant}
    For parameters $\delta, \rho \in (0,1)$, let $\mbox{Cons}_{\delta, \rho}\subset \S^{n-1}$  be the set of vectors 
    $v\in \S^{n-1}$ for which there exists a real number $\lambda$ such that $|v_i-\lambda|\leq \rho/\sqrt{n}$ holds for 
    at least $(1-\delta)n$ coordinates $i\in [n]$. A vector $v\in \S^{n-1}$ is called non-almost constant if it does 
    not belong to $\mbox{Cons}_{\delta, \rho}$.
\end{defn}

\subsection{Concentration}

Let $M$ be an $m\times n$ random matrix with i.i.d. rows, such that each row 
is uniformly drawn from $n$-dimensional $0/1$ vectors having exactly $d$ ones. 
The main result of this section is an individual concentration bound for the norm $\|Mv\|_2$, where $v$ is a fixed unit vector. We will combine this individual probability with net argument to obtain the invertibility on the set of almost constant vectors in \cref{prop: inver_cons_pcons} below.

For the case $d=n/2$, Lemma 2.9 in \cite{tran2020smallest} established such a result using Bernstein's inequality. Here we provide an alternative proof using the Paley-Zygmund inequality and Chernoff bounds for the sum of Bernoulli random variables, extending the result to any fixed $p=d/n\in (0, 1/2]$.

\begin{lemma}\label{lem: invert_indiv}
     Let $v\in \S^{n-1}$ be a fixed vector and $p\in (0, 1/2]$ be a fixed constant. Let $d\leq n$ be positive integers such that $d=pn$. There exist positive constants $C_{\ref{lem: invert_indiv}}, c_{\ref{lem: invert_indiv}}$ depending only on $p$ such that the following holds. Let $M$ be a random $m\times n$ matrix, $n/2\le m\le n$, whose rows are independent random vectors uniformly distributed on the set of $\{0, 1\}^n$ with  exactly $d$ ones. Then
    \[
    \P\lr{\|Mv\|_2\le c_{\ref{lem: invert_indiv}}\sqrt{pn}}\le e^{-C_{\ref{lem: invert_indiv}}n}.
    \]
\end{lemma}

The next lemma provides a concentration inequality for functions on the Boolean hypercube, which we will use in the proof of \cref{lem: invert_indiv}.
\begin{lemma}\cite[Lemma 2.1]{KwanSudakovTran}\label{lem: com_con_ineq}
    Let $w\in \R^n$ be a positive vector and $f: \{0,1\}^n\to \R$ be a function  such that 
    \begin{equation}\label{tran-1}
        |f(x_1, \dots, x_{i-1}, 1, x_{i+1}, \dots, x_n)-f(x_1, \dots, x_{i-1}, 0, x_{i+1}, \dots, x_n)|\le w_i,
    \end{equation}
    for all $x\in \{0,1\}^n$ and all $i\in [n]$. Suppose that $\eta$ is a random vector uniformly distributed over $\{0,1\}^n$ with exactly $d$ ones, where $1\le d \le n$.  Then
    \begin{equation*}
    \forall t \ge 0\qquad     \P(|f(\eta)-\E f(\eta)|\ge t)\le 2\exp\lr{-\frac{t^2}{8\sum_{i=1}^n w_i^2}}.
    \end{equation*}
\end{lemma}

\begin{proof}[Proof of \cref{lem: invert_indiv}]
    Let $q$ be a random vector distributed uniformly on the set of $\{0,1\}^n$ vectors with exactly $d$ ones.  
    Define the function $f: \{0,1\}^n\to \R$ by  $f(x)=\inn{x, v}$ for $x\in \{0,1\}^n$. Then  (\ref{tran-1}) holds with $w_i=|v_i|$.  Let $\mu:=\E f(q)$ and $\sigma^2:=\mbox{Var} f(q)$. 
    Applying \cref{lem: com_con_ineq} and using $v\in \S^{n-1}$, we have
   \begin{equation}\label{eq: invert_indiv}
       \forall t \ge 0\qquad       \P\lr{\abs{f(q)-\mu}\ge t}\le 2\exp\lr{-\frac{t^2}{8}}.
   \end{equation}

    Fix $\delta=\sqrt{8\log 4}$ and consider two cases based on the value of $\mu$. 

\bigskip 

\noindent 
{\it Case 1: $|\mu|\ge 2\delta$.}  By the triangle inequality, we have
\[
|f(q)|\ge |\mu|-|f(q)-\mu|\ge 2\delta-|f(q)-\mu|.
\]
Thus if  $|f(q)-\mu|\le \delta$ then  
$|f(q)|\ge \delta.$
Therefore, using \eqref{eq: invert_indiv}
    \[
    \P\lr{|f(q)|\ge \delta}\ge \P\lr{\abs{f(q)-\mu}\le \delta}\ge 1- 2\exp\lr{-\frac{\delta^2}{8}}= \frac{1}{2}.
    \]
    
\noindent 
{\it Case 2: $|\mu|<2\delta$.}  We first compute the mean and the variance. Note that we have
\begin{equation*}
    \mu=\E\lr{\sum_{i=1}^nq_iv_i}=\sum_{i=1}^nv_i\E(q_i)=\frac{d}{n}\sum_{i=1}^nv_i.
\end{equation*}
The variance is computed using $\mbox{Var} q_i=p(1-p)$ and $\mbox{Cov}(q_i, q_j)=-\frac{p(1-p)}{n-1}$ for $i\neq j$:
\begin{align*}
    \sigma^2&=\sum_{i=1}^n v_i^2\, \mbox{Var} \, q_i+\sum_{i\neq j}v_iv_j \mbox{ Cov}(q_i, q_j)=p(1-p)\sum_{i=1}^nv_i^2-\frac{p(1-p)}{n-1}\sum_{i\neq j}v_iv_j\\
    &=p(1-p)-\frac{p(1-p)}{n-1}\lr{\lr{\sum_{i=1}^n v_i}^2-\sum_{i=1}^n v_i^2}=p(1-p)-\frac{p(1-p)}{n-1}\lr{(\mu/p)^2-1}\\
    &=p(1-p)\lr{1-\frac{\mu^2/p^2-1}{n-1}}.
\end{align*}
Since $|\mu|<2\delta$ and $\delta=\sqrt{8\log 4}$ is a constant, $(\mu/p)^2$ is bounded by a constant depending only on $p$. 
Without loss of generality, we may assume that $n$ is sufficiently large (bounded below by a constant depending only on $p$). Indeed, otherwise the conclusion of the lemma follows by making $c_{\ref{lem: invert_indiv}}$ and $C_{\ref{lem: invert_indiv}}$ small enough. 
Then the term $(\mu^2/p^2-1)/(n-1)$ is bounded above by $1/2$. Therefore,
\[
\sigma^2\ge \frac{1}{2}\, p(1-p).
\]
Then we have
\[
\E \, f(q)^2=\mu^2+\sigma^2\ge \frac{1}{2}\, p(1-p).
\]
Hence, by the Paley-Zygmund inequality,
\[
\P\lr{f(q)^2>\frac{1}{4}\, p(1-p)}\ge
\P\lr{f(q)^2>\frac{1}{2}\E\lr{f(q)^2}}\ge \frac{1}{4}\frac{\lr{\E\, f(q)^2}^2}{\E\, f(q)^4}.
\]
Thus it remains to bound $\E f(q)^4$ from above. 
Applying  \eqref{eq: invert_indiv} and the distribution formula, we observe 
$$\E\, \abs{f(q)-\mu}^4 = \int_0^\infty 4u^3\, \P\lr{|f(q)-\mu|\ge u} \, du\le 128.$$ 
Using  $(a+b)^4\le 16(a^4+b^4)$ for every $a, b\in \R$ and $|\mu|\le 2\delta$,
\[
\E\, f(q)^4=\E\lr{\abs{f(q)-\mu+\mu}^4}\le 16\, \E\lr{f(q)-\mu}^4+16\mu^4\le 2^{11}+256\delta^4.
\]
Let $C_1'=2^{11}+256\delta^4$. Note that $C_1'$ is an absolute constant.
The Paley-Zygmund inequality yields
\[
\P\lr{f(q)^2>\frac{1}{4}\, p(1-p)}\ge \frac{1}{4}\frac{\lr{\E[f(q)^2]}^2}{\E[f(q)^4]}\ge \frac{1}{4}\cdot\frac{\lr{p(1-p)/2}^2}{C_1'}.
\]
Since $1-p\ge 1/2$, 
\[
\P\lr{|f(q)|\ge \sqrt{p/8}}\ge \frac{(p(1-p))^2}{16C_1'} =: c\in (0, 1/2].
\]

\smallskip 

 Thus, combining these two cases and using $\sqrt{p/8}\le 1/4\le \delta$, we obtain
\begin{equation}\label{eq: indi_smallba}
    \P\lr{|f(q)|\ge \sqrt{p/8}}\ge c. 
\end{equation}

Let $q_1, \dots, q_m$ be the independent rows of  $M$, and let $f_i=f(q_i)=\inn{q_i, v}$. Define the indicator variables $\eta_i:=\mathbbm{1}_{\{|f_i|>\sqrt{p/8}\}}$ for $i\in [m]$. From \eqref{eq: indi_smallba}, we have $$\P(\eta_i=1)=\P(|f(q_i)|>\sqrt{p/8})\ge c.$$ Let $G:=\sum_{i=1}^m\eta_i$. Then
\[
\E G =\sum_{i=1}^m \E(\eta_i)\ge cm.
\]
Since $\eta_i$ are independent Bernoulli random variables, we can apply the Chernoff bound (see, e.g., \cite[Theorem 1.1]{dubhashi2009concentration} with $\varepsilon=1/2$), obtaining 
\[
\P\lr{G\le \frac{1}{2}\E G}\le \exp\lr{-\frac{1}{8}\E(G)}.
\]
Since $\E(G)\ge cm \ge cn/2$, we observe
\[
\P\lr{G\le \frac{cn}{4}}\le \P\lr{G\le \frac{1}{2}\, \E G}\le  \exp\lr{-\frac{1}{8}\, \E G}
\le \exp\lr{-\frac{cm}{8}}\le \exp\lr{-\frac{cn}{16}}.
\]
Finally, consider the norm $\|Mv\|_2^2=\sum_{i=1}^m f_i^2$. If the event $\{G>cn/4\}$ occurs, then
\[
\|Mv\|_2^2=\sum_{i=1}^mf_i^2\ge \sum_{i: \eta_i=1}f_i^2\ge \sum_{i: \eta_i=1}\lr{\sqrt{\frac{p}{8}}}^2\ge \frac{cn}{4}\, \frac{p}{8}.
\]
Therefore,
\[
\P\lr{\|Mv\|_2>c_0\sqrt{pn}}\ge \P\lr{G>\frac{cn}{4}}\ge 1-\exp\lr{-\frac{cn}{16}},
\]
where $c_0=\sqrt{c/32}$. This  completes the proof.
\end{proof}

The next lemma provides a bound on the operator norm of the centered random matrix $\|M-\E M\|$. We first note that the 
expectation $\E M$ is the matrix where each entry is $d/n=p$, hence $\|\E M\|=d$, and that the direct calculations show that 
$$\|M-\E M \|=\|M_{|\mathcal{H}}\|,$$
where 
\begin{equation}\label{ortogsubs}
\mathcal{H}=\{x \in \R^n: \sum_{i=1}^{n}x_i=0\}
\end{equation}
(the subspace orthogonal to the  vector $\mathbf{1}_n=(1,1,\ldots,1)$).
For $d=n/2$, Tran \cite[Proposition~2.8]{tran2020smallest} used Bernstein's inequality and a net argument to bound the operator norm $\|M|_{\mathcal{H}}\|$,  finding it to be typically $O(\sqrt{n})$ (see also \cite[Lemma 5.1]{jain2021approximate}). 
This type of result, bounding $\|M-\E M\|$, holds more generally for $p=d/n\in (0, 1/2]$ being a fixed constant as stated below. 
Since the proof follows the same lines, we refer the interested reader to \cite[Proposition~2.8]{tran2020smallest} for the core
argument. Alternatively, one can use Theorem~4.1 in \cite{LLY-circ}, where such a bound is obtained with a worse 
(but enough for our purposes) probability. 

\begin{lemma}\label{lem: operatornrm_Tran}
Let $p\in (0, 1/2]$ be a fixed constant.
     Let $1\le d\le n$ be integers such that $d=pn$. Let $M$ be a random $m\times n$ matrix, $1\le m\le n$, whose rows are independent random $0/1$ vectors having exactly $d$ ones. Then there exist constants $C_{\ref{lem: operatornrm_Tran}}>0$ and $c_{\ref{lem: operatornrm_Tran}}>0$ depending only on $p$ such that 
    for all $t\ge C_{\ref{lem: operatornrm_Tran}}$ one has
    \[
    \P\lr{\|M-\E M\|\ge t\sqrt{pn}}\le 2e^{-c_{\ref{lem: operatornrm_Tran}}t^2n }.
    \]
\end{lemma}

\subsection{Anticoncentration}

Anticoncentration results provide upper bounds on the probability that a random variable falls within a small interval. 
Such results, applied to sums $\sum_{i=1}^n R_i v_i$, where $R_i$ are rows of a random matrix and $v$ is a fixed unit vector, 
are crucial for the study of the smallest singular value. Standard 
techniques often rely on the Least Common Denominator (LCD) for sums of independent random variables, as introduced in 
\cite{rudelson2008littlewood}. Several adjustments of LCD were used for different models of random matrices with independent 
entries 
(see, e.g., \cite{livshyts2021smallestTK, litvak2022singularity, fernandez}). 
However, the components of the row vectors $R_i$ in our model (uniformly sampled from 
$0/1$-vectors with $d$ ones) are not independent. To handle this dependency, Tran \cite{tran2020smallest} introduced the combinatorial Least Common Denominator (CLCD). The CLCD measures, for a given vector $x \in \R^n$, the closeness of the scaled vector 
$$(x_i-x_j)_{1 \leq i < j \leq n} \in \R^k\quad \mbox{ to } \quad \Z^k, \quad \mbox{ where } \quad k={\binom{n}{2}}.$$
The notion of the combinatorial Least Common Denominator was also used in \cite{jain2022smallest} to study the smallest singular value of  adjacency matrices for random regular digraphs.

\begin{defn}\label{def_CLCD}
    For a vector $v\in \R^n$, and parameters $\gamma\in (0,1)$, $\alpha>0$, the combinatorial Least Common Denominator (CLCD) is defined as
    \begin{equation*}
        \mbox{CLCD}_{\alpha,\gamma}(v)=\mbox{LCD}_{\alpha, \gamma}(D(v))=\inf_{\theta>0}
        \left\{\operatorname{dist}(\theta D(v), 
        \Z^k)<\min\left\{\gamma\|\theta D(v)\|_2, \alpha\right\}\right\},
    \end{equation*}
    where $k={\binom{n}{2}}$,  $D(v)$ is the vector in $\R^k$ with coordinates $v_i-v_j$ for $i<j$. 
\end{defn}

The following lemma provides a  lower bound on the CLCD for non-almost constant vectors (see \cite[Lemma 2.15]{tran2020smallest}).
\begin{lemma}\label{lem: clcd_sq_bdd}
    Let $\delta, \rho\in (0,1)$, and let $v\in \S^{n-1}\setminus \mbox{Cons}_{\delta,\rho}$. Then for every $\alpha>0$ and every $\gamma\in (0,\delta\rho/12)$, one has
    \begin{equation*}
        \mbox{CLCD}_{\alpha, \gamma}(v)\ge \sqrt{\delta n}/7.
    \end{equation*}
\end{lemma}

To quantify anticoncentration,  the L\'evy concentration function of a random variable $X$ is used. 

\begin{defn}
    For a random variable $X$ and $\varepsilon\ge 0$, the L\'evy concentration of $X$ of width $\varepsilon$ is 
    \begin{equation*}
        \Ls(X,\varepsilon)=\sup_{x\in \R} \P(|X-x|<\varepsilon).
    \end{equation*}
\end{defn}

The connection between the CLCD and anticoncentration is established in the following key result which is analogous to
standard properties of the LCD from \cite{rudelson2008littlewood}. Its proof is an adaptation of the proof of 
Theorems~1.5 and 3.2 in \cite{tran2020smallest}. We skip the details.

\begin{lemma}\label{lem:small_ball}
    Fix $b>0$ and $\gamma\in (0,1)$. Let $d\leq n$ be positive integers. 
    Let $v\in \R^{n}$ satisfy $\|D(v)\|_2\ge b\sqrt{n}$. Suppose that $\eta$ is a random vector uniformly distributed over $n$-dimensional $0/1$ vectors having  exactly $d$ ones.  Then for every $\alpha>0$ and $\varepsilon\ge 0$,
    \begin{equation*}
        \Ls\lr{\sum_{i=1}^n v_i\eta_i, \varepsilon\sqrt{\frac{d}{n}}}\le \frac{\varepsilon}{\gamma b}+\frac{1}{\gamma b}\sqrt{\frac{n}{d}}\cdot \frac{1}{\mbox{CLCD}_{\alpha,\gamma}(v)}+2e^{-4\alpha^2d/n^2}.
    \end{equation*}
\end{lemma}

\begin{remark}\label{rem:clcd_pcont}
    We apply \cref{lem:small_ball} to non-almost constant vectors $v\in \S^{n-1}$. 
    Note that for  $v\in \S^{n-1}\setminus \mbox{Cons}_{\delta, \rho}$, it follows 
    from \cite[Lemma~2.2]{litvak2019smallest} (cf., \cite[Lemma 2.2]{tran2020smallest}) 
    that 
    \[
    \|D(v)\|_2\ge \frac{\delta \rho\sqrt{n} }{4\sqrt{2}}.
    \]
  Setting $\alpha=\mu n$ in CLCD, choosing  $b=\delta \rho /4\sqrt{2}$ in \cref{lem:small_ball}, 
  and using \cref{lem: clcd_sq_bdd},   we obtain for  $0<\gamma <\delta \rho /12$
  \[
   \Ls\lr{\sum_{i=1}^n v_i\eta_i, \varepsilon\sqrt{\frac{d}{n}}}\le \frac{4\sqrt{2}\varepsilon}{\gamma \delta\rho}
   +\frac{28\sqrt{2}}{\gamma \delta^{3/2}\rho \sqrt{d}}+2e^{-4\mu^2d}.
  \]
  
\end{remark}

\section{Invertibility on almost constant vectors}\label{sec: inv_almost}

In this section, we focus on the invertibility of a random matrix $M$ on the set of almost constant vectors. The main result in this section establishes that with high probability $\|Mv\|_2$ is bounded away from zero on the set $\mbox{Cons}_{\delta, \rho}$.

\begin{prop}\label{prop: inver_cons_pcons}
Let $p \in(0, 1/2]$ be a fixed constant.
    Let $d, n$ be large enough positive integers such that $d=p n$. 
     There exist constants $C_{\ref{prop: inver_cons_pcons}}, c_{\ref{prop: inver_cons_pcons}}, \delta, \rho\in (0, 1)$ depending on $p$ only such that the following holds. Let $M$ be a random $m\times n$ matrix, $n/2\le m\le n$, whose rows are independent random  vectors drawn uniformly from the set of $n$-dimensional $0/1$ vectors having exactly $d$ ones. Then 
    \[
    \P\lr{\inf_{v\in \mbox{Cons}_{\delta, \rho} }\|Mv\|_2\le c_{\ref{prop: inver_cons_pcons}}\sqrt{pn}
    }\le e^{-C_{\ref{prop: inver_cons_pcons}}n}.
    \]
\end{prop}

To establish \cref{prop: inver_cons_pcons}, we use a net construction for almost constant vectors with respect to the 
pseudometric $d(x, y):=\|M(x-y)\|_2$, as provided by \cite[Lemma~2.12]{tran2020smallest}. Recall that $\mathcal{H}$ was 
introduced in (\ref{ortogsubs}).

\begin{lemma}\cite[Lemma 2.12]{tran2020smallest}\label{lem:net_cons_tran}
    Let $\delta, \rho \in (0, 1/12)$ and $n$ be sufficiently large with respect to $\delta$ and $\rho$. Then there is a net $\cN\subset \S^{n-1}$ of cardinality at most $\exp(2n\delta \log(5/\delta))$ such that for any $v\in \mbox{Cons}_{\delta, \rho}$ there is $w\in \cN$ so that for any deterministic $m\times n$ matrix $A$ we have
 \[ \|A(v-w)\|_2\le (\delta+2\rho)\lr{2\|A_{|\mathcal{H}}\|+\frac{\|A\|}{n}}. \]
\end{lemma}

\begin{proof}[Proof of \cref{prop: inver_cons_pcons}]
Let $C_{\ref{lem: invert_indiv}}, c_{\ref{lem: invert_indiv}}$ be the constants from \cref{lem: invert_indiv}, and $C_{\ref{lem: operatornrm_Tran}}, c_{\ref{lem: operatornrm_Tran}}$ be the constants from \cref{lem: operatornrm_Tran}. We choose $\delta, \rho \in (0, 1/12)$  sufficiently small to satisfy conditions specified later.
Our goal is to estimate the probability of the event 
   \[
   E:=\set{\inf_{v\in \mbox{Cons}_{\delta, \rho} }\|Mv\|_2\le c_{\ref{lem: invert_indiv}}\sqrt{pn}}.
   \]
Consider the following event 
\[
\cE_{\ref{lem: operatornrm_Tran}}:=\set{M\in \cM_{n,d}: \|M-\E M\|<C_{\ref{lem: operatornrm_Tran}}\sqrt{pn}}.
\]
By  \cref{lem: operatornrm_Tran}, 
\begin{equation*}
    \P(\cE_{\ref{lem: operatornrm_Tran}}^c)\le2e^{-Cn }
\end{equation*}
for the  constant $C=c_{\ref{lem: operatornrm_Tran}}C_{\ref{lem: operatornrm_Tran}}^2>0$, 
depending only on $p$.  
We split the event  $E$ based on whether $\cE_{\ref{lem: operatornrm_Tran}}$ occurs, 
as follows,
    \[E_1:=\set{\inf_{v\in \mbox{Cons}_{\delta, \rho} }\|Mv\|_2\le c_{\ref{lem: invert_indiv}}\sqrt{pn}\quad \mbox{ and } \quad \|M-\E M\|\ge C_{\ref{lem: operatornrm_Tran}}\sqrt{pn}  } \]
    and 
    \[
    E_2:=\set{\inf_{v\in \mbox{Cons}_{\delta, \rho} }\|Mv\|_2\le   c_{\ref{lem: invert_indiv}}\sqrt{pn}\quad  \mbox{ and }\quad  \|M-\E M\|< C_{\ref{lem: operatornrm_Tran}}\sqrt{pn} }.
    \]
    Then $E=E_1\cup E_2$, and
\[
\P(E)\le\P(E_1)+\P(E_2)\le \P(\cE_{\ref{lem: operatornrm_Tran}}^c)+\P(E_2)\le  \P(E_2)+2e^{-Cn }.
\]

Next  we bound $\P(E_2)$. Note, $\|\E M\|=pn$. Suppose $E_2$ occurs, then 
$$\|M-\E M\|< C_{\ref{lem: operatornrm_Tran}}\sqrt{pn}$$ 
(in particular, for large enough $n$, $\|M\|\leq n$) and there exists a vector $v\in \mbox{Cons}_{\delta, \rho}$ such that
\[
\|Mv\|_2\le c_{\ref{lem: invert_indiv}}\sqrt{pn}.
\]
Let $\cN$ be the net from \cref{lem:net_cons_tran} of cardinality at most $\exp(2n\delta \log(5/\delta))$. 
Choose $w\in \cN$ such that
\[
\|M(v-w)\|_2\le (\delta+2\rho)\lr{2\|M-\E M\|+\frac{\|M\|}{n}}\le (\delta+2\rho)(2C_{\ref{lem: operatornrm_Tran}}\sqrt{pn}+1).
\]
By the triangle inequality we observe
\begin{align*}
    \|Mw\|_2&\le \|Mv\|_2+\|M(v-w)\|_2\le c_{\ref{lem: invert_indiv}}\sqrt{pn}+(\delta+2\rho)(2C_{\ref{lem: operatornrm_Tran}}\sqrt{pn}+1)\\
    &\le c_{\ref{lem: invert_indiv}}\sqrt{pn}+(9\max\{\delta, \rho \} ) C_{\ref{lem: operatornrm_Tran}}\sqrt{pn}\le 2c_{\ref{lem: invert_indiv}}\sqrt{pn},
\end{align*}
provided that  
\begin{equation}\label{delrho}
\max\{\delta, \rho \}\le c_{\ref{lem: invert_indiv}}/(9C_{\ref{lem: operatornrm_Tran}})
\end{equation}
and $d$ is large enough. 
Thus 
$$
 E_2 \subset \left\{ \inf_{w\in \cN}\|Mw\|_2 \le 2c_{\ref{lem: invert_indiv}} \sqrt{pn} \right\}.
$$
By \cref{lem: invert_indiv}, for any fixed $w\in \cN$ we have
\[
 \P\lr{\|Mw\|_2\le 2c_{\ref{lem: invert_indiv}}\sqrt{pn}}\le e^{-C_{\ref{lem: invert_indiv}}n}.
\]
Applying the union bound over the net $\cN$, 
\[
   \P(E_2)\le \P\lr{\inf_{w\in \cN}\|Mw\|_2\le 2c_{\ref{lem: invert_indiv}}\sqrt{pn}}
   \le e^{2n\delta \log(5/\delta)} e^{-C_{\ref{lem: invert_indiv}}n/4} \le e^{-C_1n},
\]
where we assume that $\delta$  satisfies 
\begin{equation}\label{delta}
    2\delta \log(5/\delta)\le C_{\ref{lem: invert_indiv}}/8
\end{equation}
and $C_1=C_{\ref{lem: invert_indiv}}/8$. 
Thus, choosing $\delta$ and $\rho$ to be small enough to satisfy (\ref{delrho}) 
and (\ref{delta}), we obtain  the desired result. 
\end{proof}

\section{Proof of \cref{thm:upper_bdd_least}}\label{sec: proof_main}

We follow the general scheme introduced in \cite{rudelson2008least} 
(see also \cite{tatarko2018upper, tao2012topics}). 
It is based on the following simple implication:  for an invertible matrix $M$ and any $u, v>0$,
$$\exists x\in \R^n:\, \|x\|_2\le u,\,  \|(M^T)^{-1}x\|_2 \ge v \, \quad \mbox{ implies } \, \quad s_n(M)\le u /v.$$
Indeed, recall that $s_n(M)=1/\|M^{-1}\|=1/\|(M^T)^{-1}\|$, hence the existence of such $x$ means 
\[
1/s_n(M)= \|(M^T)^{-1}\|\ge \|(M^T)^{-1}x\|_2/\|x\|_2\ge v/u.
\]
Note that the condition $d =pn$ for a fixed constant $p \in (0,1/2]$ ensures that $M$ is invertible with high probability. For simplicity, we will  assume that $M$ is invertible, with the understanding that this occurs on an event of probability at least $1-2e^{-c_p n}$ (see, e.g., \cite{tran2020smallest})

Let $R_i$ denote the $i$-th row of $M$ and $X_i=R_i^T$, which we consider as  the column vector in $\R^n$. Let $H_1:=\mbox{span}\{X_2,\dots, X_n\}$. Let $P_1: \R^n\to H_1$ be the orthogonal projection onto $H_1$. Let $$x:=X_1-P_1X_1.$$
Roughly speaking, this vector $x$ quantifies the degree to which the first row $X_1$ is linearly independent from the remaining rows $X_2,\dots ,X_n$.

Note that $x$ is orthogonal to $H_1$. Since $M$ is invertible, $H_1$ has dimension $n-1$, then its orthogonal complement $H_1^\perp$ has dimension $1$. Let $f_1$ be a unit vector spanning $H_1^\perp$. Note $f_1=x/\|x\|_2$ (up to the sign). Then we have $$\|x\|_2=|\langle X_1, f_1\rangle|=\mbox{dist}(X_1, H_1).$$

\subsection{Bounding $\|x\|_2$}\label{sec: bound_x}
We first bound $\|x\|_2$ from above.
By Chebyshev's inequality, for any $u>0$
\begin{equation}\label{chebmark}
    \P(\|x\|_2\ge u)\le \frac{\E\|x\|_2^2}{u^2}.
\end{equation}
To apply this, we compute $\E\|x\|_2^2$ by conditioning on the sigma field generated by $H_1$. Denoting  the $k$-th coordinates  of $X_1$ and $f_1$ by $X_1(k)$ and $f_1(k)$ respectively, we  observe
\begin{align*}
       \E(\|x\|_2^2\, |\, H_1)&=\E(\langle  X_1, f_1\rangle^2\, |\, H_1)=\E\lr{\lr{\sum_{k=1}^n X_{1}(k)f_1(k)}^2\, |\, H_1}\\
    &=\sum_{k=1}^n\E(X_{1}(k)^2\, |\, H_1)f_1(k)^2+\sum_{i\neq j}\E(X_{1}(i)X_{1}(j)\, |\, H_1)f_1(i)f_1(j)\\
    &=\sum_{k=1}^n\E(X_{1}(k)^2)f_1(k)^2+\sum_{i\neq j}\E(X_{1}(i)X_{1}(j))f_1(i)f_1(j),
\end{align*}
where the last line follows from the fact that $X_1$ is independent of $H_1$ (spanned by  $X_2,\dots, X_n$).
Note that 
$$
 \E X_{1}(k)^2=\E X_{1}(k)=\P(X_{1}(k)=1)=d/n
$$
and  for $i\neq j$,
$$
 \E X_{1}(i)X_{1}(j)=\P(X_{1}(i)=1 \,\,\, \mbox{ and } \,\,\, X_{1}(j)=1)=\binom{n-2}{d-2}/\binom{n}{d}.
$$  
Therefore, 
\begin{align*}
    \E(\|x\|_2^2\, |\, H_1)
    &=\frac{d}{n}\sum_{k=1}^n f_1(k)^2+\frac{d(d-1)}{n(n-1)}\sum_{i\neq j}f_1(i)f_1(j)\\
    &=\frac{d}{n}\sum_{k=1}^n f_1(k)^2+\frac{d(d-1)}{n(n-1)}\lr{\lr{\sum_{k=1}^n f_1(k)}^2-\sum_{k=1}^n f_1(k)^2}\\
    &=\frac{d}{n}\cdot\frac{n-d}{n-1}+\frac{d(d-1)}{n(n-1)}\lr{\sum_{k=1}^n f_1(k)}^2,
\end{align*}
where the last equality relies on the fact that $\|f_1\|_2^2=\sum_{k=1}^n f_1(k)^2=1$.
Thus we get
\begin{equation*}
   \E \|x\|_2^2 = \E(\E(\|x\|_2^2\, |\, H_1))=\frac{d(n-d)}{n(n-1)}+\frac{d(d-1)}{n(n-1)}\, \E\lr{\sum_{k=1}^n f_1(k)}^2.
\end{equation*}
To bound $\E\lr{\sum_{k=1}^n f_1(k)}^2$, note that $f_1$ is orthogonal to $H_1$, so $\langle f_1, X_i\rangle=0$ for every $i=2,\dots, n$. Thus, 
\[
\inn{  f_1, \sum_{i=2}^nX_i}= 0.
\]
Therefore,
\begin{align*}
    \E\lr{\sum_{k=1}^n f_1(k)}^2&=\E\inn{  f_1, \mathbf{1}_n}^2=\E\lr{\inn{  f_1, \mathbf{1}_n-\frac{n}{d}\cdot\frac{1}{n-1}\sum_{i=2}^nX_i}}^2.
\end{align*}
Since $\|f\|_2=1$, 
\begin{align*}
  \E\lr{\inn{  f_1, \mathbf{1}_n-\frac{n}{d}\cdot\frac{1}{n-1}\sum_{i=2}^nX_i}}^2&\le 
   \E\left\|\mathbf{1}_n-\frac{n}{d}\cdot\frac{1}{n-1}\sum_{i=2}^nX_i\right\|_2^2,\\
    &=\sum_{j=1}^n \E\lr{1-\frac{n}{d}\cdot\frac{1}{n-1}\sum_{i=2}^nX_i(j)}^2\\
    &=\sum_{j=1}^n \frac{n^2}{(d(n-1))^2}\E\lr{\frac{d(n-1)}{n}-\sum_{i=2}^nX_i(j)}^2.
\end{align*}
Since each row $R_i$ is chosen independently, the entries $M_{2j}, \dots, M_{nj}$ in $j$-th column are independent Bernoulli random variables with 
the mean $d/n$. Thus, $\sum_{i=2}^nX_i(j)=\sum_{i=2}^nM_{ij}$ is $\mbox{Binomial}(n-1, d/n)$. Therefore,
\begin{equation*}
    \E\lr{\frac{d(n-1)}{n}-\sum_{i=2}^nX_i(j)}^2=\mbox{Var}\lr{\sum_{i=2}^nX_i(j)}=(n-1)\cdot\frac{d}{n}\lr{1-\frac{d}{n}}.
\end{equation*}
Hence,
\begin{align*}
     \E\left\|\mathbf{1}_n-\frac{n}{d}\cdot\frac{1}{n-1}\sum_{i=2}^nX_i\right \|_2^2=\frac{n(n-d)}{d(n-1)}.
\end{align*}
This implies  
\begin{equation*}
    \E(\|x\|_2^2)=\frac{d(n-d)}{n(n-1)}+\frac{(d-1)(n-d)}{(n-1)^2}=\frac{n-d}{n-1}\left(\frac{d}{n}+\frac{d-1}{n-1}\right)\le \frac{d}{n}+\frac{d-1}{n-1}\le \frac{3d}{n} .
\end{equation*}
Therefore, by (\ref{chebmark}),  for every $u>0$, 
\begin{equation}\label{eq: upper_bdd_x_1}
    \P(\|x\|_2\ge u)\le \frac{3\, d}{u^2\, n}.
\end{equation}

\subsection{Bounding $\|(M^T)^{-1}x\|_2$}
Note that
\begin{equation*}
    \|(M^T)^{-1}x\|_2=\|(M^T)^{-1}X_1-(M^T)^{-1}P_1X_1\|_2=\|e_1-(M^T)^{-1}P_1X_1\|_2.
\end{equation*}
Since $P_1X_1\in \mbox{span}\{X_2,\dots, X_n\}$, the vector $(M^T)^{-1}P_1X_1\in \mbox{span}\set{e_2,\dots, e_n}$ is orthogonal to $e_1$. Denote by $X_k^*=M^{-1}e_k$ the $k$-th column of the inverse matrix $M^{-1}$ for $k=1,\dots, n$.  We have $$\|P_1X_1^*\|_2^2=\langle P_1M^{-1}e_1, P_1M^{-1}e_1\rangle=\langle M^{-1}e_1, P_1M^{-1}e_1\rangle=\langle e_1, (M^{-1})^TP_1M^{-1}e_1\rangle=0,$$ where we used the fact that $P_1M^{-1}e_1$ belongs to $\mbox{span}\{X_2,\dots, X_n\}$, hence  $$(M^{-1})^TP_1M^{-1}e_1\in \mbox{span}\{e_2,\dots, e_n\}. $$
Therefore, denoting $Y_k:=P_1X_k^*$,
\begin{align*}
    \|(M^T)^{-1}x\|_2^2&=\|e_1\|_2^2+\|(M^T)^{-1}P_1X_1\|_2^2 
    \ge \|(M^T)^{-1}P_1X_1\|_2^2
    \\&=\sum_{k=1}^n\langle(M^T)^{-1}P_1X_1, e_k\rangle^2
    =\sum_{k=1}^n\langle X_1, P_1X_k^*\rangle^2
    \\&=\sum_{k=2}^n\langle X_1, \underbrace{P_1X_k^*}_{=:Y_k}\rangle^2=\sum_{k=2}^n\langle X_1, Y_k\rangle^2.
\end{align*}

The next lemma follows from \cite[Lemma 2.1]{rudelson2008least} which works for any invertible matrix. By the result in \cite{tran2020smallest}, $M$ is invertible with probability at least $1-2e^{-c_p n}$, therefore, we may assume without loss of generality that $M^{T}$ is invertible in our lemma. In particular, $H_1:=\mbox{span}\set{X_2,\dots, X_n}$ is  
$(n-1)$-dimensional.

\begin{lemma}\label{lem: bio_system_k2}
    Recall that $X_k=R_k^T$ is the $k$-th column vector of $M^{T}$ for $k\in [n]$, and $P_1$ is the orthogonal projection onto $H_1$. If $Y_k=P_1(M^{-1}e_k)$ for $k=2,\dots, n$ is defined as above, then $\{X_k, Y_k\}_{k=2}^n$ is a  biorthogonal system in $H_1$. 
\end{lemma}

The following is a consequence of the uniqueness in \cref{prop_biorth}.

\begin{corollary}\label{coro:unique_system}
    The system of vectors $\{Y_k\}_{k=2}^n$  defined as in \cref{lem: bio_system_k2} is uniquely determined by the system $\{X_k\}_{k=2}^n$ within the subspace $H_1$. In particular, the system $\{Y_k\}_{k=2}^n$ and the vector $X_1$ are independent. 
\end{corollary}

Denote $H_{i,j}=\mbox{span}\{X_k: k\notin\{i, j\}\}$. By \eqref{eq:bioth}, we have $\|Y_k\|_2=1/\mbox{dist}(X_k, H_{1,k})$. Therefore, 
\begin{align*}
     \|(M^T)^{-1}x\|_2^2\geq \sum_{k=2}^n \langle X_1, Y_k\rangle^2=\sum_{k=2}^n \inn{X_1, \frac{Y_k}{\|Y_k\|_2}}^2
     \, \|Y_k\|_2^2 =\sum_{k=2}^n\frac{a^2_k}{b^2_k},
\end{align*}
where we denoted 
$$a_k=\la{\langle  X_1, Y_k\rangle}/\|Y_k\|_2\quad \mbox{ and } \quad b_k=1/\|Y_k\|_2=\mbox{dist}(X_k, H_{1,k})\quad \mbox{ for }\, \, k\in \{2,\dots, n\}.$$

\subsection{Probabilistic bounds for $a_k$ and $b_k$}
 Next, we show that for each $k$, with high probability $a_k$ is bounded from below and $b_k$ is bounded from above. Without loss of generality, we will do this for $k=2$. The argument for any $k\in \{2,\dots, n\}$ is the same as $k=2$. 

We split the unit sphere into sets of almost constant vectors and non-almost constant vectors which were introduced in \cref{def:almost_constant}. First, we will show that $H_{1,2}^\perp$ consists of non-almost constant vectors with high probability. Consider an $(n-2)\times n$ matrix $B$ with rows $X_3,\dots, X_n$. Since the subspace $H_{1,2}$ is the span of random vectors $X_3,\dots, X_n$, 
we observe $H_{1,2}^\perp \subset \ker(B)$. 

By  \cref{prop: inver_cons_pcons} applied to $(n-2)\times n$ matrix $B$, there exist positive constants $C_{\ref{prop: inver_cons_pcons}}, c_{\ref{prop: inver_cons_pcons}}, \delta, \rho\in (0, 1)$ depending only on $p$, such that
\begin{equation}\label{rvscheme}
\P\lr{H_{1,2}^\perp\cap  \mbox{Cons}_{\delta, \rho}=\emptyset }\ge \P\lr{\inf_{v\in \mbox{Cons}_{\delta, \rho} }\|Bv\|_2\ge c_{\ref{prop: inver_cons_pcons}}\sqrt{pn}
    }\ge 1-e^{-C_{\ref{prop: inver_cons_pcons}}n},
\end{equation}
which exactly means that the subspace $H_{1,2}^\perp$ consists of not almost constant vectors with high probability. 

Fix an absolute constant $\mu>0$ and let $\gamma \in (0, \delta \rho/12)$. Denote 
\[
  \cE:=\set{v\in H_{1,2}^\perp\cap \S^{n-1}:\,\, \mbox{CLCD}_{\mu n, \gamma}(v)\ge \sqrt{\delta n}/7}
\]
By \cref{lem: clcd_sq_bdd} and (\ref{rvscheme}), 
\begin{equation}\label{eq: thm1_notalmostconstant}
    \P\lr{\cE}\ge  1-e^{-C_{\ref{prop: inver_cons_pcons}}n}.
\end{equation}

Conditioning on the subspace $H_{1}$, we fix a realization of vectors $\{X_i\}_{i=2}^n$. Then by the uniqueness in \cref{coro:unique_system} the vector $Y_2$ is also fixed. Denote $Y:=Y_2/\|Y_2\|_2$. By \cref{lem: bio_system_k2},  $\{X_i\}_{i=2}^n$ and  $\{Y_i\}_{i=2}^n$ form a biorthogonal system. In particular, $Y_2$ is orthogonal to  $\{X_i\}_{i=3}^n$. 
Thus, $Y\in H_{1,2}^\perp\cap\S^{n-1}$. Conditioning on the event $\cE$, we know that 
\[
    \mbox{CLCD}_{\mu n,\gamma }(Y)\ge \sqrt{\delta n}/7.
\]
By \cref{lem:small_ball} and parameters chosen as in \cref{rem:clcd_pcont}, 
for every $\varepsilon\ge 0$, there exist  constants $C_1, C_2>0$ depending on $\rho, \delta, \gamma$ such that
\[
\P\lr{a_2\le \varepsilon\sqrt{p}\, |\, X_2,\dots, X_n}\le C_1\varepsilon+\frac{C_2}{\sqrt{d}}+2e^{-4\mu^2pn}.
\]
Now we unfix all random vectors $X_2, \dots, X_n$ and then by \eqref{eq: thm1_notalmostconstant} we get for every $\varepsilon\ge 0$
\begin{align*}
    \P(a_2\le \varepsilon\sqrt{p})&=\P(\{a_2\le \varepsilon\sqrt{p}\}\cap\cE)+\P(\cE^c)\\&\le \E_{X_2,\dots, X_n}\P(\{a_2\le \varepsilon\sqrt{p}\, |\, X_2,\dots, X_n\}\cap\cE)+\P(\cE^c)\\&\le  C_1\varepsilon+\frac{C_2}{\sqrt{d}}+2e^{-4\mu^2pn}+e^{-C_{\ref{prop: inver_cons_pcons}}n}.
\end{align*}

To show a lower bound on $b_2$, we use a similar argument as in the estimate \eqref{eq: upper_bdd_x_1} for $\|x\|_2$. 
We condition on $H_{1,2}$, use Chebyshev's inequality and then take the expectation to obtain for every $t>0$, 

\[
\P\lr{b_2\ge t}=\P(\mbox{dist}(X_2, H_{1,2})\ge t)\le \frac{3\, d}{t^2\, n}.
\]
Hence,
\[
  \forall t>0 \qquad \P\lr{b_2\ge t\sqrt{3p}}\le \frac{1}{t^2}.
\]

Combining the probability estimates for $a_2$ and $b_2$, we obtain that  for every $t>0$ and every $\varepsilon\ge 0$ one has
\[
\P\lr{a_2\le \varepsilon\sqrt{p}\quad  \mbox{ or } \quad b_2\ge t\sqrt{3p}}\le  C_1\varepsilon+\frac{C_2}{\sqrt{d}}+2e^{-4\mu^2pn}+e^{-C_{\ref{prop: inver_cons_pcons}}n}+\frac{1}{t^2},
\]
where $C_1, C_2>0$ are  constants depending only on $p$. 

Repeating the above argument for all $k\in \{2,\dots, n\}$ instead of $k=2$ we conclude that for any $t>0$, $\varepsilon\ge 0$, and for large enough $d$ and a large enough constant $C_3>0$ (which may depend on $p$),
\begin{equation}\label{akbk}
\P\lr{\frac{a_k}{b_k} \le \frac{\varepsilon}{\sqrt{3}\, t}}\le C_1\varepsilon+\frac{C_3}{\sqrt{d}}+\frac{1}{t^2}. 
\end{equation}

\smallskip 

To complete the proof we  need the following general statement.

\begin{prop}\cite[Proposition 2.2]{rudelson2008least}\label{prop:Markov}
Let $Z_k$ be non-negative random variables for $k\in [n]$. Then for every $\varepsilon>0$, 
\begin{equation*}
    \P\lr{\frac{1}{n}\sum_{k=1}^nZ_k\le \varepsilon}\le \frac{2}{n}\sum_{k=1}^n\P(Z_k\le 2\varepsilon).
\end{equation*}
\end{prop}

We  use \cref{prop:Markov} for $Z_k=(a_k/b_k)^2$ to get
\[
\P\lr{ \|(M^T)^{-1}x\|_2\le \frac{\varepsilon \sqrt{n}}{\sqrt{3}\, t}}\le \P\lr{ \frac{1}{n}\sum_{k=2}^nZ_k\le \frac{\varepsilon^2}{3t^2}}\le \frac{2}{n}\sum_{k=2}^n\P\lr{Z_k\le \frac{2\varepsilon^2}{3t^2}}.
\]
By (\ref{akbk}) this implies 
\[
\P\lr{ \|(M^T)^{-1}x\|_2 \le \frac{\varepsilon \sqrt n}{\sqrt{3}\, t}} \le  2\sqrt{2}C_1\varepsilon+\frac{2C_3}{\sqrt{d}}+\frac{2}{t^2}.
\]
Combining  the above estimate with \eqref{eq: upper_bdd_x_1}, we obtain that
for every $\varepsilon> 0, u>0, t>0$,
\begin{align*}
    \P\lr{s_n(M)\le \frac{ut}{\varepsilon}\sqrt{\frac{p}{n}}}&\ge 
    \P\lr{\|x\|_2\le u\sqrt{p}/\sqrt{3}\quad \mbox{ and } \quad \|(M^T)^{-1}x\|_2\ge 
    \frac{\varepsilon\sqrt{n}}{\sqrt{3}\, t}}\\
    &\ge 1-C_4\lr{\varepsilon+\frac{1}{\sqrt{d}}+\frac{1}{t^2}+\frac{1}{u^2}},
\end{align*}
where $C_4>0$ is a constant depending only on $p$. 
To complete the proof, we fix $\varepsilon>0$, choose $u=t=1/\sqrt{\varepsilon}$, and adjust the constants.



\bibliographystyle{abbrv}
\bibliography{Singular_value}




\vspace{1cm}

\par\noindent
\textsc{Alexander E. Litvak} 
\\
\textit{Email address:} \texttt{alitvak@ualberta.ca}

\vspace{0.25cm}

\par\noindent
\textsc{Dongbin Li} 
\\
\textit{Email address:} \texttt{dongbin@ualberta.ca}

\vspace{0.25cm}

\par\noindent
\textsc{Tingzhou Yu} 
\\
\textit{Email address:} \texttt{tingzho1@ualberta.ca}

\vspace{0.5cm} 
\par\noindent
Department of Mathematical and Statistical Sciences, University of Alberta, Edmonton, AB, T6G 2R3, Canada.

\end{document}